\documentclass[11pt]{amsart}
\usepackage{soul}
\usepackage[latin1]{inputenc}
\usepackage{ amssymb, mathtools, mathrsfs,xypic,bm}
\usepackage{enumerate}
\usepackage[numbers]{natbib}
\usepackage{setspace}
\usepackage{color}

\newtheorem{theorem}{Theorem}[section] 
\newtheorem{lemma}[theorem]{Lemma}     
\newtheorem{corollary}[theorem]{Corollary}
\newtheorem{proposition}[theorem]{Proposition}
\newtheorem{definition}[theorem]{Definition}
\newtheorem{example}[theorem]{Example}

\newcommand{\ep}{\varepsilon}

\newcommand{\N}{\mathbb{N}}
\newcommand{\Z}{\mathbb{Z}}

\newcommand{\C}{\mathbb{C}}

\newcommand{\K}{\mathbb{K}}

\newcommand{\Aut}[1]{{\rm Aut}(#1)}

\newcommand{\id}[1]{{\rm id}_{#1}}

\newcommand{\QH}{{\textrm{(QH)}}}
\newcommand{\wA}{\widetilde{A}}

\begin{document}
\title[Deformations of Wreath Products]{Deformations of Wreath Products} 

\author{Marius Dadarlat}
\address{MD: Department of Mathematics, Purdue University, West Lafayette, IN 47907, USA}
\email{mdd@purdue.edu}

\author{Ulrich Pennig}
\address{UP: School of Mathematics, Cardiff University, Senghennydd Road, Cardiff CF24 4AG, UK}
\email{PennigU@cardiff.ac.uk}	

\author{Andrew Schneider}
\address{AS: Department of Mathematics, Purdue University, West Lafayette, IN 47907, USA}
\email{schneid5@purdue.edu}

\thanks{M.D. was partially supported by NSF grant \#DMS--1362824}
\thanks{U.P. was partially supported by the SFB 878 -- ``Groups, Geometry \& Actions'' while employed at the University of M\"unster}

\begin{abstract}
 Connectivity is a homotopy invariant property of a separable $C^*$-algebra $A$ which has three important consequences: absence of nontrivial projections, quasidiagonality and realization of the Kasparov group $KK(A,B)$ as homotopy classes of asymptotic morphisms from $A$ to $B\otimes \K$ if  $A$ is nuclear.
Here we give a new characterization of connectivity for separable exact C*-algebras  and use this characterization to show that the class of discrete countable amenable groups whose augmentation ideals are connective is closed under generalized wreath products. In a related circle of ideas, we give a result on  quasidiagonality of reduced crossed-product C*-algebras associated to noncommutative Bernoulli actions.
\end{abstract}

\maketitle

\section{Introduction}
Voiculescu \cite{Voi:unitaries} has shown that the K-theory of the two-torus $\mathbb{T}^2$ can be captured from sequences of pairs of  almost commuting unitaries $u_n,v_n \in U(n)$ with $\lim_{n\to \infty}\|u_nv_n-v_nu_n\|=0$ or equivalently from completely positive and contractive (cpc) discrete asymptotic morphisms $\{\varphi_n:C(\mathbb{T}^2)\cong C^*(\Z^2)\to M_n(\C)\}_n$.
Let us recall that a cpc discrete asymptotic morphism is a sequence of completely positive contractive maps $\varphi_n \colon A \to B_n$, which is almost multiplicative in the sense that $\lim_{n \to \infty}\lVert \varphi_n(a)\varphi_n(b) - \varphi_n(ab) \rVert = 0$ for all $a,b\in A$. Voiculescu's example is not an isolated phenomenon. Indeed, Connes and Higson \cite{Con-Hig:etheory} showed that the concept of asymptotic morphism plays a fundamental role in the algebraic topology of $C^*$-algebras. The homotopy classes of asymptotic morphisms
from the suspension of $A$ to the stable suspension of $B$ is isomorphic to the E-theory group $E(A,B)\cong  [[SA,SB\otimes \K]]$,  the universal half-exact $C^*$-stable  homotopy bifunctor on separable $C^*$-algebras.
 Building on these ideas, Houghton-Larsen and Thomsen \cite{Thomsen-Larsen} have shown
 that the Kasparov groups can be realized as homotopy classes of cpc asymptotic morphisms,  $KK(A,B)\cong [[SA,SB\otimes \K]]^{cp}$. A cpc asymptotic morphism consists of a family of cpc maps $\{\varphi_t:A \to B\}_{t\in [1,\infty)}$  such that the map $t\mapsto \varphi_t(a)$ is continuous and
 $\lim_{t \to \infty}\lVert \varphi_t(a)\varphi_t(b) - \varphi_t(ab) \rVert = 0$ for all $a,b\in A$.
 The role of suspensions is two-fold as it provides both a group structure and a good supply of maps due to the quasidiagonality of $SA$.

 An important question in this context is to characterize the class of C*-algebras for which one can dispense with suspensions and realize $E(A,B)$ and $KK(A,B)$ as homotopy classes of asymptotic morphisms $[[A,B \otimes \K]]$  and respectively $[[A,B \otimes \K]]^{cp}$. Desuspension results  have played a key role in the classification theory of nuclear C*-algebras  \cite{Ror:encyclopedia}.
 Moreover, the realization of K-homology of a C*-algebra $A$ as homotopy classes of cpc deformations of $A$ into matrices $[[A, M_{\infty}(\C)]]^{cp}\cong [[A, \K]]^{cp} $  has other significant applications as illustrated
  in \cite{CGM:flat}, \cite{Dadarlat-almost}.  The pairing $K_0(A)\times K^0(A) \to \Z$ can then be described using the canonical trace on matrices rather than the Fredholm index.

A first answer  to the  question of desuspending in E-theory is given in \cite{DadLor:unsusp}: the natural map $[[A, B \otimes \K]] \to E(A,B)$ is an isomorphism for all separable $C^*$-algebras $B$ if and only if $A$ is \emph{homotopy symmetric}, which means that $[[\id{A}]] \in [[A, A \otimes \K]]$ has an additive inverse or equivalently that $[[A, A \otimes \K]]$ is a group.
Unfortunately,  it is quite hard in practice to check that a  given $C^*$-algebra is homotopy symmetric. In a recent paper \cite{Dad-Pennig-homotopy-symm},  we employed results of Thomsen \cite{Thomsen:discrete} to prove that a separable nuclear C*-algebra is homotopy symmetric if and only if $A$ is \emph{connective}, a property which is much easier to verify, see Definition~\ref{def:connectivity} below.
Our original terminology for connectivity was \emph{property} {\QH}. In this paper we give a new characterization of this property, see Prop.~\ref{new-characterization}, and that prompted us to introduce the more descriptive notion of connectivity in place of property {\QH}. A countable discrete group $G$ is called connective if  the kernel $I(G) \subset C^*(G)$  of the trivial representation $\iota:C^*(G)\to \C$ is a connective C*-algebra.

Connectivity is a homotopy invariant property and it has the important feature that it passes to  $C^*$-subalgebras. This has allowed us to exhibit vast new  classes of homotopy symmetric $C^*$-algebras \cite{Dad-Pennig-homotopy-symm}.
 Connectivity  has two other important consequences: absence of nontrivial projections and quasidiagonality.
With this in mind, the task of establishing  connectivity for large classes of (amenable) group C*-algebras becomes particularly interesting, since   this stronger property  would give a new explanation of   why the conjectures of Kadison-Kaplansky and  Rosenberg
are true for amenable groups as proved by Higson and Kasparov \cite{HigKas:BC} and respectively by Tikuisis, White and Winter \cite{TWW:quasidiagonality}.

  It is implicitly conjectured in \cite{Dadarlat-almost} that the augmentation ideal $I(G)$ of  a discrete, torsion free,  amenable group $G$  is homotopy symmetric and hence that the Kasparov group $KK(I(G),B)$ can be realized as the homotopy classes of asymptotic morphisms $[[I(G),B \otimes \K]]$ for any separable $C^*$-algebra $B$. The case of abelian groups is covered by  results from \cite{DadLor:unsusp}.
The conjecture has been verified for nilpotent groups in \cite{Dad-Pennig-homotopy-symm} using the equivalence between homotopy-symmetry and connectivity. It was also shown there that the class of discrete countable connective amenable groups is closed under torsion free central extensions and under direct limits.
The main result of this paper is Theorem~\ref{thm:wreath-product} which shows that the class of discrete countable amenable connective groups is closed under generalised wreath products. Since connectivity passes to subgroups, this class contains a lot of new examples of connective groups, including the free solvable ones.
Moreover, Corollary~\ref{corollary:periodic} shows that semidirect products of amenable discrete connective groups with respect to periodic actions are connective. The arguments from the proof of Theorem~\ref{thm:wreath-product} lead naturally to the question of quasidiagonality of crossed products of the type
$(\bigotimes_G D)\rtimes_r G$ for $D$ a separable unital C*-algebra and $G$ a discrete countable group. Taking advantage of the breakthrough results from \cite{ORS:qd_elem_amen} and \cite{TWW:quasidiagonality} we show in Theorem ~\ref{thm:Bernoulli}  that $(\bigotimes_G D)\rtimes_r G$ is quasidiagonal if and only if $D$ is quasidiagonal and $G$ is amenable.

\section{Preliminaries}
We will use the notation from \cite[Sec.\ 5]{Dad-IJM}. For a Hilbert space $\mathcal{H}$,
we denote by $L(\mathcal{H})$ the C*-algebra
of bounded and linear operators on $\mathcal{H}$. The ideal of compact operators is denoted by $\K$.
If $A$ is a $C^*$-algebra,
$\mathcal{H}$, $\mathcal{H}'$ are  Hilbert spaces, $F \subset A$ is a finite set, $\ep > 0$ and
$\varphi \colon A \to L(\mathcal{H})$ and $\psi \colon A \to L(\mathcal{H}')$ are two maps,
we write $\varphi \prec_{F,\ep} \psi$ if there is an isometry $v :\mathcal{H}\to \mathcal{H}'$
such that $\lVert \varphi(a) - v^*\psi(a)v \rVert < \ep$ for all $a \in F$.
If $v$
can be chosen to be a unitary, we write $\varphi \sim_{F,\ep} \psi$. Moreover, we
write $\varphi \prec \psi$ if $\varphi \prec_{F,\ep} \psi$ for all finite sets
$F$ and for all $\ep > 0$.
Most maps that we use in this paper are either unital and completely positive (abbreviated ucp) or
completely positive and contractive (cpc).
 If $\{\varphi_n \colon A \to L(\mathcal{H}_n)\}_n$
and $\{\varphi'_n \colon A \to L(\mathcal{H}'_n)\}_n$ are two sequences of maps, we write
$(\varphi_n)\sim (\varphi'_n)$ if there is a sequence of unitaries $u_n:\mathcal{H}_n \to \mathcal{H}_n'$
such that $\lim_{n\to \infty}\|\varphi_n(a)-u_n^*\varphi'(a)u_n\|=0$ for all $a\in A$.
A ucp (or cpc) asymptotic morphism is a sequence $\{\varphi_n \colon A \to B_n\}_n$ of ucp (respectively cpc) maps which are asymptotically multiplicative in the sense that $\lim_{n\to \infty}\|\varphi_n(ab)-\varphi_n(a)\varphi_n(b)\|=0$ for all $a,b \in A$.

 Let us recall from \cite{Dad-Pennig-homotopy-symm} that a separable $C^*$-algebra $A$ has property {\QH} if there is a discrete cpc asymptotic homomorphism $\{\gamma_n:A \to L(\mathcal{H}_n)\}_n$ with $\dim(\mathcal{H}_n)=k_n \nearrow \infty$ ,
which is injective and null-homotopic. This means that $\limsup_{n}\|\gamma_n(a)\|=\|a\|$ for all $a\in A$ and that there is a discrete cpc asymptotic homomorphism
$\{\varphi_n:A \to C_0[0,1)\otimes L(\mathcal{H}_n)\}_n$, $\varphi_n=(\varphi_n^{(t)})_{t\in [0,1]}$, such that
$\varphi_n^{(0)}=\gamma_n$ for all $n\in \N$.
It is shown in \cite[Prop.2.5]{Dad-Pennig-homotopy-symm} that in the definition of property {\QH} restated above, one can replace all the spaces $\mathcal{H}_n$ by the same separable infinite dimensional Hilbert space $\mathcal{H}$. Let us denote   $C_0[0,1)\otimes L(\mathcal{H})$ by $CL(\mathcal{H})$.
The following definition will change terminology from property {\QH} to connectivity. This is motivated by Proposition~\ref{new-characterization}.
\begin{definition}\label{def:connectivity}
 A separable C*-algebra $A$ is \emph{connective} if
there is a  $*$-monomor\-phism $\Phi:A\to\prod_n CL(\mathcal{H})/\bigoplus_n CL(\mathcal{H})$ which is liftable to a cpc map $\varphi: A\to\prod_n CL(\mathcal{H})$.
A discrete countable group $G$ is \emph{connective} if its augmentation ideal $I(G)=\mathrm{ker}(\iota:C^*(G)\to \C)$ is connective.
\end{definition}
By \cite[Prop.2.5]{Dad-Pennig-homotopy-symm}, $A$ has property {\QH} if and only if $A$ is connective.
A connective C*-algebra is quasidiagonal and $p=0$ is the only idempotent element of $A\otimes \K$ and in fact of any minimal tensor product $A\otimes B$.
Indeed, it is straightforward to check that these properties are inherited from $CL(\mathcal{H})$.

 Proposition~\ref{new-characterization} below gives a new equivalent definition of connectivity, in the case of exact $C^*$-algebras. Specifically, it shows that all the components of the injective and null-homotopic discrete cpc asymptotic homomorphism  $\{\gamma_n\}_n$ can be chosen to be  equal to any given $*$-representation of $A$ on $\mathcal{H}$ which is essential, i.e.~$\pi^{-1}(K(\mathcal{H}))=\{0\}$.
For the proof we need  Lemma~5.3 of \cite{Dad-IJM} reproduced below.
\begin{lemma}[\cite{Dad-IJM} ]\label{lem:discrete_inverse}
Let $B$  be a separable unital $C^*$-algebra.
Let  $\{\varphi_n \colon B \to M_{k(n)}(\C)\}_n$
and $\{\gamma_n \colon B \to M_{r(n)}(\C)\}_n$ be ucp discrete asymptotic morphisms. Suppose that
 $\limsup_{n}\|\gamma_n(b)\|=\|b\|$ for all $b\in B$. Then there exist a sequence $(\omega(n))$ of disjoint
finite subsets of $\N$ with $\max \omega(n-1) < \min \omega(n)$ and a ucp discrete asymptotic
morphism  $\{\varphi^{\,\prime}_n \colon B \to M_{s(n)}(\C)\}_n$ such that
$\left(\varphi_n \oplus \varphi^{\,\prime}_n\right) \sim (\gamma_{\omega(n)})$, where
$\gamma_{\omega(n)} = \oplus_{i \in \omega(n)} \gamma_i$.
\end{lemma}

For a $C^*$-algebra $A$ we denote by $\widetilde{A}$ its unitalization.
Let  $\iota:\wA \to \C$ be the corresponding character.
Let $\mathcal{H}$ be a separable Hilbert space. The map obtained by composing $\iota$ with the
unital homomorphism $\C \to L(\mathcal{H})$ will be denoted $\iota\cdot \mathrm{id}_{\mathcal{H}}$ or by  $\iota^\infty$ if $\mathcal{H}$ is infinite dimensional. {Let $\gamma \colon \widetilde{A} \to L(\mathcal{H})$ be a ucp map. We will use the notation $\gamma^{\infty}$ for the infinite sum $\gamma \oplus \gamma \oplus \dots$.

\begin{proposition}\label{new-characterization}
Let $A$ be a separable exact $C^*$-algebra. Then $A$ is connective if and only if
 for any  essential unital representation $\pi:\wA \to L(\mathcal{H})$ of $\wA$ on a separable infinite dimensional  Hilbert space, any finite subset $F\subset \wA$ and any $\ep>0$ there is an $(F,\ep)$-multiplicative  ucp map $\varphi:\wA \to C[0,1]\otimes L(\mathcal{H})$, $\varphi=(\varphi_t)_{t\in [0,1]}$
such that $\varphi_0=\pi$ and $\varphi_1=\iota^\infty$.
\end{proposition}
\begin{proof} One direction is trivial and it holds for arbitrary separable C*-algebras $A$.
Suppose now that $A$ is exact and connective. By unitalizing the relevant $C^*$-algebras and maps,
we obtain a
 ucp discrete asymptotic morphism $\{\Gamma_n \colon \wA \to C[0,1]\otimes L(\mathcal{H}_n) \}_n$,
$\Gamma_n=(\Gamma^{(t)}_n)_{t\in [0,1]}$, $\mathcal{H}_n$ finite dimensional, such that $\{\Gamma^{(0)}_n \colon \wA \to L(\mathcal{H}_n) \}_n$
is injective and $\Gamma^{(1)}_n=\iota\cdot \mathrm{id}_{\mathcal{H}_n}$ for all $n\in \N$. Let $\gamma_n:=\Gamma^{(0)}_n$.

If $\alpha_n:\wA \to L(\mathcal{H}_n)$ is a sequence of maps and $m>n$ we define $\alpha_{[n,m]} =\alpha_n \oplus \alpha_{n+1}\oplus \dots\oplus \alpha_{m}$ .

\textbf{Claim}.  Fix an essential representation $\pi:\wA \to L(\mathcal{H})$. Then, for any finite subset $F'\subset \wA$, any $\ep'>0$ and any $n_0\in \N$,  there are integers $m>n>n_0$
such that  $\pi\sim_{F',\ep'} \gamma_{[n,m]} ^\infty$.

Let us observe that if we prove the claim, then we can complete the proof as follows.
Let $F$ and $\ep$ be given as in the statement of the proposition. It is straightforward to find a finite set $F'\subset \wA$ and $\ep'>0$
with the property that
 for any two $(F',\ep')$-multiplicative ucp maps $\alpha,\beta:\wA \to L(\mathcal{H})$  with $\|\alpha(a)-\beta(a)\|<\ep'$ for all $a\in F'$,
 it follows that $(1-t)\alpha+t\beta$ is $(F,\ep)$-multiplicative for all $t\in [0,1]$.
Fix $n_0$ such that $\Gamma_n$ is $(F,\ep)$-multiplicative for all $n> n_0$ {and choose $m,n$ as in the claim}.
We may identify the Hilbert space $\mathcal{H}$ with the Hilbert space corresponding to $\gamma_{[n,m]} ^\infty$.
Since $\pi\sim_{F',\ep'} \gamma_{[n,m]} ^\infty$, there is  a unitary $u\in U(\mathcal{H})$ such that $\|u\pi(a)u^*-\gamma_{[n,m]} ^\infty(a)\|<\ep'$ for all $a\in F'$.
Let $u_t$ be a {norm-}continuous path of unitaries in $U(\mathcal{H})$ from $1$ to $u$.
Define $\varphi_t$ to be the continuous path of ucp maps obtained by concatenating the paths
$\pi_t=u_t\pi u^*_t$ with $\mu_t=(1-t)u\pi u^*+t\gamma_{[n,m]}^\infty$ and with $\nu_t=(\Gamma^{(t)}_{[n,m]})^\infty$. Due to our choice of $F'$, $\ep'$ and $n_0$, the map $t\mapsto \varphi_t$ defines a continuous path of $(F,\ep)$-multiplicative ucp maps
joining $\pi$ with $\iota^\infty$.

Let us now verify the {claim} for any given data $\pi$,  $F'$,  $\ep'$ and  $n_0$.
By \cite[Lemma 5.1]{Dadarlat:MRL} applied to $\wA$, $F'\subset \wA$  and $\ep'>0$, there exist a finite subset $F''\subset\wA$ and $\ep''>0$ with the property that if $\alpha_i:\wA\to L(\mathcal{H}_i)$, $i=1,2$ are any two $(F'',\ep'')$-multiplicative ucp maps such that $\alpha_1^\infty\prec_{F'',\ep''} \alpha_2$ and
$\alpha_2^\infty\prec_{F'',\ep''} \alpha_1$ then $\alpha_1\sim_{F',\ep'}\alpha_2$.
  Let $(F_i)_i$ be an increasing sequence of finite subsets of $\wA$ with union  dense in $\wA$ and let $(\ep_i)_i$  be a sequence of strictly positive  numbers convergent $0$. Since $A$ is separable, exact and quasidiagonal, it follows from  Theorem 6 of \cite{Dad;apqd} that there is a discrete ucp asymptotic morphism
$\{\theta_i:\wA \to L(\mathcal{H}_i)\}_{i\in \N}$, with $\mathcal{H}_i$ finite dimensional Hilbert spaces,
such that $\pi\sim_{F_i,\ep_i}\theta_i^\infty$ for all $i$.
Fix now $i$ sufficiently large such that $\pi\sim_{F'',\ep/2''}\theta_i^\infty$.

By Lemma ~\ref{lem:discrete_inverse}, there exist integers $m>n>n_0$ such that $\gamma_{[n,m]}$ is $(F'',\ep'')$-multiplicative and $\theta_i \prec_{F'',\ep''/2} \gamma_{[n,m]}$ and hence $\theta_i ^\infty \prec_{F'',\ep''/2} \gamma_{[n,m]}^\infty$.
This last  relation in conjunction with $\pi\sim_{F'',\ep''/2}\theta_i^\infty$ gives $\pi^\infty\prec_{F'',\ep''} \gamma_{[n,m]}^\infty$.
By Voiculescu's theorem \cite{Voi:Weyl-vn} and Stinespring's theorem \cite[Thm.~1.5.3]{Brown-Ozawa} we have $\gamma_{[n,m]}^\infty\prec_{F'',\ep''} \pi$.  We can then apply \cite[Lemma 5.1]{Dadarlat:MRL} as explained above, to conclude that
$\pi\sim_{F',\ep'} \gamma_{[n,m]}^\infty$.
\end{proof}

\begin{lemma} \label{lem:ses}	
Let $G$ and $H$ be discrete groups and assume that $H$ acts on $G$ by automorphisms. Then there is a split short exact sequence
\[
	0 \to I(G) \rtimes H \to I(G \rtimes H) \to I(H) \to 0\ ,
\]
where we use the maximal crossed product throughout.
\end{lemma}

\begin{proof}
By the universality property of the crossed product \cite[p.170]{KirWas:e-bundles}, the sequence
\[
	0 \to I(G) \rtimes H \to C^*(G) \rtimes H \to \C \rtimes H \to 0
\]
is exact. There are natural isomorphisms $\C \rtimes H \cong C^*(H)$ and $C^*(G) \rtimes H \cong C^*(G \rtimes H)$. Let $J$ be the kernel of the map $I(G \rtimes H) \to I(H)$ induced by the projection $G \rtimes H \to H$. The rows in the commutative diagram
\[
	\xymatrix{
		0 \ar[r] & J \ar[r] \ar[d] & I(G \rtimes H) \ar[d] \ar[r] & I(H) \ar[d] \ar[r] & 0 \\
		0 \ar[r] & I(G) \rtimes H \ar[r] & C^*(G \rtimes H) \ar[d] \ar[r] & C^*(H) \ar[r] \ar[d] & 0 \\
		& & \C \ar[r]_-{=} & \C
	}
\]
are exact and a diagram chase yields that the map $I(G \rtimes H) \to C^*(G \rtimes H)$ restricts to an isomorphism $J \cong I(G) \rtimes H$. The sequence splits via the $*$-homomorphism induced by the splitting $H \to G \rtimes H$.
\end{proof}

\begin{proposition} \label{prop:crossed_prod}	
Let $G$ and $H$ be countable discrete amenable groups, such that $H$ acts on $G$ by automorphisms. Suppose that $H$ and $I(G) \rtimes H$ are connective. Then $G \rtimes H$ is connective.
\end{proposition}

\begin{proof}
This follows from Lemma~\ref{lem:ses} together with \cite[Thm.~3.3 (d)]{Dad-Pennig-homotopy-symm}
which states that a split extension of separable nuclear connective $C^*$-algebras  is connective.
\end{proof}

\section{Wreath Products}
For a unital $C^*$-algebra $D$ and a countable set $J$ one defines the minimal tensor product $\bigotimes_J D$
as the inductive limit of finite  minimal products $\bigotimes_F D$ where $F$ runs through the family of finite subsets of $J$ ordered by inclusion. Note that if $F\subset F'$ the connecting map $\bigotimes_F D\to \bigotimes_{F'} D$ is isometric. Indeed this map takes $\bigotimes_{j \in F} a_j$ to $\bigotimes_{j' \in F'} b_{j'}$
where $b_{j'}=a_j$ if $j\in F$ and $b_{j'}=1_D$ if $j'\in F'\setminus F$.

Let $D$ be a unital separable $C^*$-algebra endowed with a unital character $\iota:D \to \C$ and let $I(D)=\mathrm{ker}(\iota)$. Let $\iota_{\otimes} \colon \bigotimes_J D \to \C$ be the character induced by $\iota$ and let
$I(\bigotimes_J D)$ denote its kernel. Let $H$ be a discrete countable group and let
 $J$ be a set with a left action of $H$.
 $H$ acts on $A:= \bigotimes_{J} D$ by permuting the tensor factors: $h \cdot (\bigotimes_{j \in J} a_{j}) = \bigotimes_{j \in J} a_{h^{-1}j}$. We denote this action by $\alpha \colon H \to \Aut{A}$.
Note that $I(\bigotimes_J D)$ is an $H$-invariant ideal of $A$ { since $\iota_{\otimes}\circ \alpha_h=\iota_{\otimes}$ for all $h\in H$.}

The following proposition gives new examples of connective $C^*$-algebras.
\begin{proposition}\label{prop-algebras} Suppose that {$D$ is exact and that}  $I(D)$ is connective. Then for any discrete {countable} {amenable} group $H$ the  crossed product $I(\bigotimes_J D)\rtimes H$ is connective.
\end{proposition}
\begin{proof} Let $\pi:D \to L(\mathcal{H})$ be a faithful essential representation.
Since $I(D)$ is connective, by Proposition~\ref{new-characterization} there exists a discrete ucp asymptotic morphism
\(
	\{{\varphi}_n \colon D \to C[0,1] \otimes L(\mathcal{H})\}_n
\), such that  ${\varphi}_n^{(0)} = {\pi}$  and ${\varphi}_n^{(1)} = \iota\cdot \mathrm{id}_{\mathcal{H}}$ for all $n\in \N$.  Let $B = \bigotimes_{J} L(\mathcal{H})$
(minimal tensor product) and define $\Phi_n^{(t)} = \bigotimes_{J} \varphi_n^{(t)} \colon A \to B$ for each $t \in [0,1]$.  By Stinespring's theorem, $\Phi_n^{(t)}$ is also a ucp map. For each $n \in \N$ and $a \in A$ the map $t \mapsto \Phi_n^{(t)}(a)$ is continuous. In order to verify this property, since $\Phi_n^{(t)}$ is linear and contractive, we may assume without any loss of generality  that $a = \bigotimes_{j \in F} a_{j} \in \bigotimes_F D \subset A$ for some finite set $F \subset J$. Continuity of $t \mapsto \Phi_n^{(t)}(a)$ follows in this case from the continuity of $t \mapsto \bigotimes_{j \in F} \varphi_n^{(t)}(a_{j})$. Hence, we obtain a sequence of ucp maps
\[
	\{\Phi_n \colon A \to C[0,1] \otimes B\}_n\ .
\]

We show that  this sequence is in fact an asymptotic morphism. To prove asymptotic multiplicativity of $\{\Phi_n\}_n$ we need to verify that $\lim_{n\to \infty}\|\Phi_n(ab)-\Phi_n(a)\Phi_n(b)\|=0$ for $a,b \in A$. Since the $\Phi_n$ are linear and contractive, we may restrict to the case $a = \bigotimes_{j \in F} a_{j}$ and $b = \bigotimes_{j \in F} b_{j}$ for a finite subset $F \subset J$ with $\|a_j\|, \|b_j\|\leq 1$ for all $j\in F$. Then, for each $t\in [0,1]$ we have
\begin{align*}
	&\ \, \lVert \Phi^{(t)}_n(a)\Phi^{(t)}_n(b) - \Phi^{(t)}_n(ab)  \rVert \\
	= & \ \left\lVert \bigotimes_{j \in F}  \varphi^{(t)}_n(a_{j})\varphi^{(t)}_n(b_{j}) - \bigotimes_{j \in F}\varphi^{(t)}_n(a_{j} b_{j}) \right\rVert \\
	\leq & \sum_{j \in F} \lVert \varphi^{(t)}_n(a_{j})\varphi^{(t)}_n(b_{j}) - \varphi^{(t)}_n(a_{j}b_{j}) \rVert,
\end{align*}
and hence \[\lVert \Phi_n(a)\Phi_n(b) - \Phi_n(ab)  \rVert \leq  \sum_{j \in F} \lVert \varphi_n(a_{j})\varphi_n(b_{j}) - \varphi_n(a_{j}b_{j}) \rVert.\]
Each term of the sum from above  converges to $0$ for $n \to \infty$.
Thus, the ucp maps $\Phi_n \colon A \to C[0,1] \otimes B$ form a discrete asymptotic morphism such that $\Phi_n^{(0)}=\bigotimes _J \pi$ is an isometric $*$-homomorphism and $\Phi_n^{(1)} = \iota_{\otimes} \cdot 1_B$.

 Let $\beta \colon H \to \Aut{B}$ be the action of $H$ on $B$ that permutes the tensor factors: $\beta_h (\bigotimes_{j \in J} b_{j}) = \bigotimes_{j \in J} b_{h^{-1}j}$.
 Then we have $\Phi_n^{(t)} \circ \alpha_{h} = \beta_{h} \circ \Phi_n^{(t)}$ for all $h\in H$ and $t\in [0,1]$ or equivalently
\(	\Phi_n \circ \alpha_{h} = \gamma_{h} \circ \Phi_n,\)
where $\gamma_h=id_{C[0,1]}\otimes\beta_{h}$.
Since $H$ acts trivially on $C[0,1]$, there is a canonical isomorphism of crossed products $(C[0,1] \otimes B) \rtimes H \cong C[0,1] \otimes (B \rtimes H)$ (see for example \cite[Ex.~4.1.3]{Brown-Ozawa}). By \cite[Thm.~3.5~(d)]{paper:Raeburn}, ${\Phi}_n$ induces a ucp map:
\[
	\widetilde{\Phi}_n \colon A \rtimes H \to C[0,1] \otimes (B \rtimes H)
\]
 such that for all $a\in A$ and $h\in H$, $\widetilde{\Phi}_n(a u_h)={\Phi}_n(a) v_h$, where we denote by $u_h$ and $v_h$ the canonical unitaries of the corresponding crossed products so that
  $\alpha_h(a)=u_h a u_h^*$ and
 $\gamma_h(b)=v_h b v_h^*$.  Let $a,a' \in A$ and $h,{h'}\in H$. The identities
\begin{align} \label{eqn:asy_mult}
	 & \lVert \widetilde{\Phi}_n(a u_h)\widetilde{\Phi}_n(a' u_{h'}) - \widetilde{\Phi}_n(au_h\,a'u_{h'}) \rVert \\
	=\,& \lVert \Phi_n(a)\,\gamma_h(\Phi_n(a')) v_{hh'}- \Phi_n(a\,\alpha_{h}(a')) v_{hh'}\rVert \notag \\
	=\,& \lVert \Phi_n(a)\,\Phi_n(\alpha_h(a')) - \Phi_n(a\,\alpha_h(a')) \rVert \notag
\end{align}
show that the sequence $\widetilde{\Phi}_n$ is asymptotically multiplicative.
By Theorem~7.7.5 from \cite{Ped:C*-aut},
any equivariant embedding of $C^*$-algebras induces an embedding of reduced crossed products.
It follows  that $\widetilde{\Phi}_n^{(0)} \colon A \rtimes H \to  B \rtimes H$ is an isometric $*$-homomorphism since ${\Phi}_n^{(0)}$ has that property.
Let
\[
	\eta_n = \left.\widetilde{\Phi}_n\right|_{I(A) \rtimes H} \colon I(A) \rtimes H \to C[0,1] \otimes (B \rtimes H).
\]
We have that $\eta_n^{(0)}$ is an isometric $*$-homomorphism, and $\eta_n^{(1)} = 0$ since $\widetilde{\Phi}_n^{(1)}(au_h) = \Phi_n^{(1)}(a)v_h=\iota_{\otimes}(a) v_h=0$ for all $a\in I(A)=\mathrm{ker}(\iota_{\otimes})$ and $h\in H$.
Finally, since $I(A) \rtimes H$ is separable, there is a separable $C^*$-algebra $E\subset B \rtimes H$ such that
the image of $\eta_n$ is contained in $C_0[0,1)\otimes E$ for all $n$.
Since $\eta_n^{(0)}$ is an isometric $*$-homomorphism,  it follows that $I(A) \rtimes H$ is connective, since Definition~\ref{def:connectivity} is verified as a consequence of the
 ``if\,'' part of Proposition~\ref{new-characterization}.\end{proof}
Let $G$ and $H$ be countable discrete groups and let $J$ be a set with a left action of $H$. Recall that the wreath product is defined as
\[
	G \wr_{\substack{{ }\\ J}} H = \left(\bigoplus_{J} G\right) \rtimes H\ ,
\]
where $H$ acts on the direct sum via the action on the indices.

\begin{theorem} \label{thm:wreath-product}
Let $G$ and $H$ be countable discrete amenable  groups  and let $J$ be a countable $H$-set. If $G$ and $H$ are connective, then the wreath product $G \wr_{\substack{{ }\\ J}} H$ is also connective.
\end{theorem}

\begin{proof}
Let $A = \bigotimes_{J} C^*(G)$. Since $C^*(\bigoplus_{J} G) \cong \bigotimes_{J} C^*(G)$, we have that $C^*(G \wr_{\substack{{ }\\ J}} H) \cong A \rtimes H$, where $H$ acts on $A$ by permuting the tensor factors: $h \cdot (\bigotimes_{j \in J} a_{j}) = \bigotimes_{j \in J} a_{h^{-1}j}$.
Let $\iota \colon C^*(G) \to \C$ be the character induced by the trivial representation and let $\iota_{\otimes} \colon A \to \C$  be the character induced by $\iota$. Note that the isomorphism $\bigotimes_{J} C^*(G) \cong C^*(\bigoplus_{J} G)$ intertwines $\iota_{\otimes}$ with the corresponding character of $C^*(\bigoplus_{J} G)$. Let
\[
	I(A) = \ker(\iota_{\otimes}) \cong I\!\left(\bigoplus_{J} G \right)\ .
\]
 We apply Proposition \ref{prop-algebras} to obtain that $I(A) \rtimes H$ is connective.
 It follows now from Proposition~\ref{prop:crossed_prod} that $G \wr_{\substack{{ }\\ J}} H$ is connective.
\end{proof}

As a consequence of the last theorem we can prove that semidirect products with respect to periodic actions are connective.
\begin{corollary}\label{corollary:periodic}
	Let $G$, $H$ be countable discrete amenable connective groups. Let $\alpha \colon H \to \Aut{G}$ be a homomorphism with finite image. Then the semidirect product $G \rtimes_{\alpha} H$ is connective.
\end{corollary}

\begin{proof}
Let $K = \ker(\alpha)$ and set $S = H / K \cong \text{im}(\alpha)$. The group $S$ is finite by assumption. $H$ acts on $S$ via left multiplication. Denote this action by $\beta$. Let $\pi \colon H \to S$ be the quotient homomorphism and let $\dot{\alpha} \colon S \to G$ be such that $\dot{\alpha} \circ \pi = \alpha$. Consider the group homomorphism
\[
	\varphi \colon G \to \bigoplus_S G
\]
given by $\varphi(g)(s) = \dot{\alpha}_{s^{-1}}(g)$. Observe that $\varphi$ is injective and
\begin{align*}
\varphi(\alpha_h(g))(s) &= \dot{\alpha}_{s^{-1}}(\alpha_h(g)) = \dot{\alpha}_{s^{-1}\pi(h)}(g) =	 \dot{\alpha}_{(\pi(h)^{-1}s)^{-1}}(g) \\
&= \varphi(g)(\pi(h)^{-1}s) = \beta_h(\varphi(g))(s)\ ,
\end{align*}
which proves that $\varphi$ intertwines the two actions on both sides. Hence we obtain a homomorphism $\Phi \colon G \rtimes_{\alpha} H \to G \wr_S H$. Since  connectivity passes to $C^*$-subalgebras, in view of Thm.~\ref{thm:wreath-product} it suffices to prove that $\Phi$ is injective. Since $\Phi(g,h) = (\varphi(g),h)$, this follows from the injectivity of $\varphi$.
\end{proof}

\begin{example} {\rm
Since  connectivity passes to $C^*$-subalgebras, it follows from Thm.~\ref{thm:wreath-product} that if
 $G$ and $H$ are countable discrete amenable connective groups, then any subgroup of  $G \wr H = G \wr_H H$ also is connective

In this way we obtain many interesting examples of connective groups, including the free solvable groups $S_{r,n}$ on $r$ generators of derived length $n$. Every solvable group with $r$ generators of derived length~$n$ is a quotient of $S_{r,n}$. The groups  $S_{r,n}$ can be defined recursively as follows: Let $F^{(0)} = F=F_r$ be the free group on $r$~generators and let $F^{(n+1)} = [F^{(n)}, F^{(n)}]$ be the commutator subgroup of $F^{(n)}$. Then we have $S_{r,n} = F/F^{(n)}$. In particular, $S_{r,1} \cong \Z^r$ and $S_{r,2}$ is the free metabelian group on $r$~generators. In other words $S_{r,2}$ is a metabelian group with $r$~generators which maps surjectively onto any other metabelian group with $r$~generators.
By Magnus' embedding theorem \cite{Mag:embedding} we have the following: Let $G$ be a finitely generated countable discrete group and choose a free group $F = F_r$ on $r$ generators and a normal subgroup $R \triangleleft F$, such that $G$ is isomorphic to $F/R$. Let $R'$ be the commutator subgroup of $R$. Let $A$ be the free abelian group on $r$ generators. Then $F/R'$ embeds as a subgroup into  $A \wr G$. Thus, if $G$ is amenable and is connective, the same holds true for $F/R'$. By induction it follows from this observation that all $S_{r,n}$ is connective.

Chapuis gave a characterization of the subgroups of $\Z^r \wr \Z^s$ in terms of model theory \cite[Cor.~3.1]{Chapuis:free_meta} generalizing the embedding theorem of Magnus in the free metabelian case. He called these $\forall$-free metabelian. By our observations above, they are all connective.}
\end{example}

\section{Quasidiagonality of crossed-products associated to noncommutative Bernoulli shifts}
Prop.~\ref{prop-algebras} leads naturally to the question of quasidiagonality of the crossed products of the type
$(\bigotimes_G D)\rtimes_r G$ for $D$ a separable unital C*-algebra and $G$ a discrete countable group. Theorem ~\ref{thm:Bernoulli} below shows that $(\bigotimes_G D)\rtimes_r G$ is quasidiagonal if and only if $D$ is quasidiagonal and $G$ is amenable.

Recall the following characterisation of quasidiagonality due to Voiculescu. A separable $C^*$-algebra $A$ is quasidiagonal if and only if for every finite subset $F \subset A$ and every $\varepsilon > 0$ there is $n \in \N$ and a cpc map $\psi \colon A \to M_n(\C)$ such that $\lVert \psi(xy) - \psi(x)\psi(y) \rVert < \varepsilon$ for all $x,y \in F$ and $\lVert \psi(x) \rVert \geq \lVert x\rVert - \varepsilon$ for all $x \in F$ \cite[Thm.~1]{Voi:qd}. Note that in fact it suffices to fulfil the condition $\lVert \psi(x) \rVert \geq \lVert x\rVert - \varepsilon$ for just a single element $x\in F$ at a time, since one can then consider finite directs sums of such maps.

All the tensor products in the sequel are minimal tensor products.
\begin{lemma} \label{lem:tensor_dah}
Let $A$, $B$, $C_n$, $D_n$ be separable $C^*$-algebras. If $\{\alpha_n \colon A \to C_n\}_n$ and $\{\beta_n \colon B \to D_n\}_n$ are cpc discrete asymptotic morphisms with
\begin{gather*}
	\limsup_{n \to \infty} \lVert \alpha_n(a) \rVert = \lVert a\rVert\ , \quad
	\lim_{n \to \infty} \lVert \beta_n(b) \rVert = \lVert b\rVert \ , 	
\end{gather*}
then $\{\alpha_n \otimes \beta_n \colon A \otimes B \to C_n \otimes D_n\}_n$ is a cpc discrete asymptotic morphism such that $\limsup_{n \to \infty} \lVert (\alpha_n \otimes \beta_n)(x) \rVert = \lVert x \rVert$ for all $x \in A \otimes B$.
\end{lemma}

\begin{proof}
	As in the proof of Prop.~\ref{prop-algebras} one verifies that $\alpha_n \otimes \beta_n$ is a cpc discrete asymptotic morphism.
	Therefore it suffices to prove that the induced $*$-homomorphism
	\[
		\eta \colon A \otimes B \to \frac{\prod_{n \in \N} C_n \otimes D_n}{\bigoplus_{n \in \N} C_n \otimes D_n}
	\]
	is injective. Seeking a contradiction suppose this is false. Let $J = \ker(\eta) \neq 0$. By Kirchberg's Slice Lemma \cite[Lem.~4.1.9]{Ror:encyclopedia} there is $0 \neq x \in A \otimes B$, such that $x^*x \in J$ and $xx^* = a \otimes b$ for some $a \in A$, $b \in B$. We have
	\begin{align*}
		0 &= \lVert \eta(x^*x) \rVert = \lVert \eta(xx^*) \rVert = \limsup_{n \to \infty} \lVert \alpha_n(a) \otimes \beta_n(b) \rVert \\
		&= \limsup_{n \to \infty} \left(\lVert \alpha_n(a) \rVert\,\lVert \beta_n(b) \rVert\right) = \left(\limsup_{n \to \infty}\lVert \alpha_n(a) \rVert\right) \lVert b \rVert = \lVert a \rVert\,\lVert b \rVert\ .
	\end{align*}
	But this implies $\lVert a \otimes b \rVert = 0$ and therefore $x = 0$, which is a contradiction.
\end{proof}

Using the results from \cite{ORS:qd_elem_amen} and \cite{TWW:quasidiagonality}, we can now adapt the method used in Prop.~\ref{prop-algebras} to prove the following:

\begin{theorem}\label{thm:Bernoulli}
	Let $D$ be a unital separable quasidiagonal $C^*$-algebra and let $G$ be a countable discrete amenable group. Then the crossed product  $(\bigotimes_G D)\rtimes G$
is quasidiagonal.
\end{theorem}
Here we work with minimal tensor products and $G$ acts via noncommutative Bernoulli shifts.
Rosenberg has shown that quasidiagonality of the reduced group $C^*$-algebra $C^*_r(G)$ of a countable discrete group $G$ implies
the amenability of $G$ \cite{Had:QD_amenable}. Hence, since both $D$ and $C^*(G)$ are subalgebras of $(\bigotimes_G D)\rtimes_r G$, the conditions that
$D$ is quasidiagonal and $G$ is amenable are certainly necessary.
\begin{proof}
	Let $A = \bigotimes_G D$. By Voiculescu's characterisation of quasidiagonality it suffices to find a discrete ucp asymptotic morphism $\{\Psi_n \colon A \rtimes G \to B_n\}_n$ with $\limsup_{n \to \infty} \lVert \Psi_n(x) \rVert = \lVert x \rVert$ for all $x \in A \rtimes G$, such that all $C^*$-algebras $B_n$ are quasidiagonal.

	By quasidiagonality of $D$ there is a ucp discrete asymptotic morphism $\{\varphi_n \colon D \to M_{k(n)}(\C)\}_n$ with $\lim_{n \to \infty} \lVert \varphi_n(d) \rVert = \lVert d \rVert$ for all $d \in D$. Define
	\[
		\Phi_n \colon \bigotimes_G D \to \bigotimes_G M_{k(n)}(\C)
	\]
	by $\Phi_n = \bigotimes_G \varphi_n$ and let $R_n = \bigotimes_G M_{k(n)}(\C)$. As in the proof of Prop.~\ref{prop-algebras} it follows that $\left\{\Phi_n\right\}_{n \in \N}$ is a ucp discrete asymptotic morphism. Next we show that $\limsup_{n \to \infty} \lVert \Phi_n(a) \rVert = \lVert a \rVert$ for all $a \in A$. It suffices to verify this for elements $a \in \bigotimes_F D$ for any finite set $F \subset G$. This is proved by induction on  the cardinality of $F$, using  Lemma~\ref{lem:tensor_dah}.
	
	Since each $\Phi_n$ is cpc and $G$-equivariant (where $G$ acts on $R_n$ via Bernoulli shifts), each $\Phi_n$ extends to a ucp map $\tilde{\Phi}_n \colon A \rtimes G \to R_n \rtimes G$ by \cite[Thm.~3.5~(d)]{paper:Raeburn}. The same calculation as in (\ref{eqn:asy_mult}) shows that $\{\tilde{\Phi}_n\}_{n \in \N}$ is a discrete asymptotic morphism. By  \cite{ORS:qd_elem_amen} and \cite{TWW:quasidiagonality} the algebra $B_n = R_n \rtimes G$ is quasidiagonal.
	
	To conclude the proof it remains to be shown that
	\begin{equation} \label{eqn:asy_inj}
		\limsup_{n \to \infty} \lVert \tilde{\Phi}_n(y) \rVert = \lVert y \rVert \qquad \forall y \in A \rtimes G\ .
	\end{equation}
	To this purpose consider the commutative diagram
	\[
		\xymatrix{
			A \rtimes G \ar[d]_-{E} \ar[r]^-{\tilde{\Phi}_n} & R_n \rtimes G \ar[d]^-{E_n} \\
			A \ar[r]_-{\Phi_n} & R_n
		}
	\]
	where $E$ and $E_n$ are the canonical faithful conditional expectations (see \cite[Prop.~4.1.9]{Brown-Ozawa}).
	Suppose (\ref{eqn:asy_inj}) is false. Then there exists an $x \in A \rtimes G$ with $x \geq 0$, $x \neq 0$ such that $\limsup_{n \to \infty} \lVert \tilde{\Phi}_n(x) \rVert = 0$. But since $E_n(\tilde{\Phi}_n(x)) = \Phi_n(E(x))$ and $E_n$ are contractive, we obtain
	\[
		0 = \limsup_{n \to \infty} \lVert E_n(\tilde{\Phi}_n(x)) \rVert = \limsup_{n \to \infty} \lVert \Phi_n(E(x)) \rVert = \lVert E(x) \rVert\ ,
	\]
	which implies $x = 0$ by faithfulness of $E$ and yields a contradiction.
\end{proof}

\emph{Acknowledgements.} We thank the referee for a number of suggestions that improved the exposition and for pointing out that Lemma~\ref{lem:ses} and its proof holds without assuming amenability of the involved groups.  Part of this work was completed during the research program ``Classification of Operator Algebras: Complexity, Rigidity, and Dynamics'' at the Institut Mittag-Leffler. The authors would like to thank the organisers and the staff at the IML for the hospitality and for the inspiring and productive atmosphere.

\bibliographystyle{abbrv}


\end{document}